\documentclass[article,12pt]{amsart}

\usepackage{amsfonts} 
\usepackage{amsmath}
\usepackage{amssymb}
\usepackage{amsthm}
\usepackage{setspace}
\usepackage{subfigure}
\usepackage{url}
\usepackage{booktabs}
\usepackage{ifthen}
\usepackage{tikz}
\usetikzlibrary{calc}
\usetikzlibrary{decorations.pathreplacing}
\usepackage{enumerate}
\usepackage{enumitem}

\newtheorem{thm}{Theorem}[section]
\newtheorem{lem}[thm]{Lemma}
\newtheorem{prop}[thm]{Proposition}

\newtheorem{defn}[thm]{Definition}
\newtheorem{alg}[thm]{Algorithm}

\newcommand{\Z}{\mathbb{Z}}

\numberwithin{equation}{section}

\begin{document}
 
\title{All trees are six-cordial}

\author{Keith Driscoll}
\address{Keith Driscoll (\tt keithdriscoll@clayton.edu)}

\author{Elliot Krop}
\address{Elliot Krop (\tt elliotkrop@clayton.edu)}

\author{Michelle Nguyen}
\address{Michelle Nguyen (\tt ngannguyen@clayton.edu)}
\address{Department of Mathematics, Clayton State University}

\date{\today}

\begin{abstract}
For any integer $k>0$, a tree $T$ is $k$-cordial if there exists a labeling of the vertices of $T$ by $\mathbb{Z}_k$, inducing edge-weights as the sum modulo $k$ of the labels on incident vertices to a given edge, which furthermore satisfies the following conditions: 
\begin{enumerate}
\item  Each label appears on at most one more vertex than any other label.
\item Each edge-weight appears on at most one more edge than any other edge-weight.
\end{enumerate}
Mark Hovey (1991) conjectured that all trees are $k$-cordial for any integer $k$. Cahit (1987) had shown earlier that all trees are $2$-cordial and Hovey proved that all trees are $3,4,$ and $5$-cordial. We show that all trees are six-cordial by an adjustment of the test proposed by Hovey to show all trees are $k$-cordial.
\\[\baselineskip] 
\\[\baselineskip]
\end{abstract}

\maketitle

\section{Introduction}
All graphs will be finite and simple. For basic graph theoretic notation and definitions, we refer the reader to D. West \cite{West}. For a survey of graph labeling problems and results, see Gallian \cite{Gallian}. 

For any tree $T$ and any integer $k>0$, a \emph{k-cordial} labeling of $T$ is a function $f:V(T)\rightarrow \Z_k$ inducing an edge-weighting also denoted by $f$, defined by $f(uv)=f(u)+f(v) \pmod k$ for any edge $(uv)$ of $T$, which satisfy the following conditions:
\begin{enumerate}[label=(\roman*)]
\item Each label appears on at most one more vertex than any other label.
\item Each weight appears on at most one more edge than any other edge-weight.
\end{enumerate}

In other words, if for any $a\in \Z_k$ we define $v_a$ and $e_a$ as the number of vertices and edges, respectively, which are labeled $a$, then the above conditions can be rewritten as
\begin{enumerate}[label=(\roman*)]
\item $|v_a-v_b|\leq 1$ for any distinct $a,b\in \Z_k$
\item $|e_a-e_b|\leq 1$ for any distinct $a,b\in \Z_k$
\end{enumerate}

Cahit \cite{Cahit} was the first to define $2$-cordial labelings (which he called cordial) as a simplification of graceful and harmonious labelings. Motivated by the Graceful Tree Conjecture of Rosa \cite{Rosa} and the Harmonious Tree Conjecture (HTC) of Graham and Sloane \cite{GrahamSloane}, he showed that all trees are $2$-cordial. The extension of the definition to groups is due to Hovey \cite{Hovey} who also showed that all trees are $3,4,$ and $5$-cordial. Hovey conjectured that trees are $k$-cordial for all $k$ and that all graphs are $3$-cordial. In the last twenty-five years there has been little  progress towards a solution to either conjecture. However, Cichacz, G\"orlich, and Tuza \cite{CGT} extended this problem to hypergraphs while Pechenik and Wise \cite{PW} considered the existence of cordial labelings for the smallest non-cyclic group $V_4$.

It should be noted that a solution to Hovey's first conjecture implies the HTC.

\section{Definitions and Facts}

The following two simple properties can be found in \cite{GrahamSloane}.

\begin{lem}\label{add}
For any $k>0$, if $f$ is $k$-cordial labeling of a graph $G$, then $f+a$ is a $k$-cordial labeling of $G$ for any $a\in \mathbb{Z}_k$.
\end{lem}

\begin{lem}\label{unitmult}
For any $k>0$, if $f$ is $k$-cordial labeling of a graph $G$, then $-f$ is a $k$-cordial labeling of $G$.
\end{lem}

\begin{defn}
A \emph{caterpillar} is a tree $T$ such that for a maximum path $P$, all vertices are of distance at most one from $P$.
\end{defn}

The next result can be found in \cite{Hovey}.

\begin{thm}\label{cat}
Caterpillars are $k$-cordial for all $k>0$.
\end{thm}

The proof of the above theorem is obtained by the sequential labeling of Grace \cite{Grace}. We include its description since we use variants of this labeling throughout the paper.

\begin{alg}{$k$-cordial labelings of caterpillars}

Given a caterpillar $T$, draw $T$ as a planar bigraph with partite sets $A$ and $B$. Choose any nonnegative integer $\ell$ and label the vertices of $A$ sequentially by $\ell, \dots, \ell+|A|$, starting at the top (bottom). Next label the vertices of $B$ sequentially by $\ell+|A|+1, \dots, \ell +|A|+|B|$ starting at the top (bottom). Reduce all labels modulo $k$ to obtain a $k$-cordial labeling.

\end{alg}

The following is consequence of Grace's algorithm:

\begin{prop}\label{fact2}
If $T$ is a rooted caterpillar with $k$ vertices (not counting the root) for some positive integer $k$, longest path $P$, and root $r$ at distance $1$ from an endpoint of $P$, then $T$ is $k$-cordial with every weight appearing exactly once.
\end{prop}

\begin{proof}
We label the vertices of $T$ by drawing the caterpillar as a bipartite graph and applying Grace's algorithm begining by labeling the root $0$. Since the root is either the first or last vertex of its part, proceed in Grace's algorithm by labeling the rest of the vertices in the part sequentially. The only repeated label will be zero on the final vertex in the part not containing the root. Since the labeling is sequential, all weights are represented.
\end{proof}

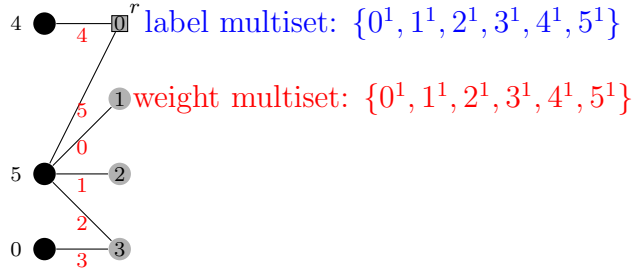
\begin{figure}[ht]
\begin{center}
\begin{tikzpicture}[scale=1]
\tikzstyle{vert}=[circle,fill=black,inner sep=3pt]
\tikzstyle{overt}=[circle,fill=black!30, inner sep=3pt]
\tikzstyle{root}=[rectangle,draw,fill=black!30,inner sep=3pt]

  \node[vert, label=left:\tiny{0}] (u1) at (2.5,-5) {};
  \node[overt, label=center:\tiny{3}] (u2) at (3.5,-5) {};
  \node[vert, label=left:\tiny{5}] (u3) at (2.5,-4) {};
  \node[root, label=center:\tiny{0}] (u4) at (3.5,-2) {};
  \node[overt, label=center:\tiny{2}] (u6) at (3.5,-4) {};  
  \node[overt, label=center:\tiny{1}] (u7) at (3.5,-3) {};
  \node[vert, label=left:\tiny{4}] (u8) at (2.5,-2) {};

  \draw[color=black] 
   (u1)--(u2)--(u3)--(u4)
   (u6)--(u3) (u7)--(u3) (u8)--(u4) 
;

\node [color=black] at (3.70,-1.80){\tiny{$r$}};

  \node [color=red] at (3,-5.15) {\tiny{3}};
  \node [color=red] at (3,-4.65) {\tiny{2}};
  \node [color=red] at (3,-4.15) {\tiny{1}};
  \node [color=red] at (3,-3.65) {\tiny{0}};
  \node [color=red] at (3,-3.15) {\tiny{5}};
  \node [color=red] at (3,-2.15) {\tiny{4}};

\node at (7,-2){\color{blue}label multiset: $\{0^1,1^1,2^1,3^1,4^1, 5^1\}$};
\node at (7,-3){\color{red}weight multiset: $\{0^1,1^1,2^1,3^1,4^1,5^1\}$};

\end{tikzpicture}
\caption{Rooted Caterpillar as in Proposition \ref{fact2}}
\end{center}
\end {figure}

The following fact can be easily verified.

\begin{prop}\label{fact1}
Every tree on at most six vertices is a caterpillar.
\end{prop}

Hovey \cite{Hovey} defined $A$-cordiality of rooted forests for any abelian group $A$ and studied the particular case when $A$ is cyclic. In all but one instance, we apply his definition for the special case when the rooted forest is a rooted tree. Note that in the following definition, the root set is not considered a subset of the vertex set of the rooted forest. Thus, a rooted forest $F$ of order $k$ has $k$ vertices and a set of roots $R$ that are not the vertices of $F$ and where the size of $R$ is the same as the number of components of $F$. 

\begin{defn}
For any integer $k>0$, a rooted forest $F$ with vertex set $V$, edge set $E$, and root set $R$ is $k$-cordial if for every labeling $g:R\rightarrow \Z_k$ and every $\ell\in \Z_k$, there is a function $f:V\rightarrow \Z_k$ satisfying 
\begin{enumerate}
\item $|v_i - v_j|\leq 1$ for all $i,j\in \Z_k$ 
\item $|e_i - e_j|\leq 1$ for all $i,j\in \Z_k$ where neither $i$ nor $j$ is $\ell$
\item $0\leq e_{\ell}-e_i\leq 2$ for all $i\in \Z_k$
\end{enumerate}
\end{defn}

\begin{thm}[Hovey \cite{Hovey}]\label{test}
For any $k>0$, if all trees and rooted forests with $k$ vertices are $k$-cordial, then all trees are $k$-cordial.
\end{thm}

\begin{lem}[Hovey \cite{Hovey}]\label{hovey}
If all trees on $mk$ vertices are $k$-cordial, then so are all trees $T$ with $mk\leq |T|\leq mk + \lfloor\frac{k}{2}\rfloor+1$.
\end{lem}

\begin{proof}
Attaching a leaf to a tree with $mk+j$ vertices allows for $k-j$ labels on the pendant vertex and forbids $j-1$ weights on the pendant edge, when $j>0$. Such a labeling exists whenever $k-j>j-1$. When $j=0$ any label is allowed and there is a choice for a weight.
\end{proof}

\begin{defn}
For any tree $T$ we say $T$ is \emph{split} at roots $v_1,\dots, v_m$ if $T$ can be written as a union of a tree $T_0$ and rooted trees $T_1, \dots, T_m$, for some positive integer $m$,  with roots $v_1,\dots, v_m$, so that for any distinct $i,j\in \{0,\dots, m\}, T_i \cap T_j$ is either $v_i$ or $v_j$.
\end{defn}

\begin{figure}[ht]
\begin{center}
\begin{tikzpicture}[scale=1]
\tikzstyle{vertex}=[circle, draw, inner sep=0pt, minimum size=6pt]
\tikzstyle{vert}=[circle,fill=black,inner sep=3pt]
\tikzstyle{overt}=[circle,fill=black!30, inner sep=3pt]
\tikzstyle{root}=[rectangle,fill=black,inner sep=3pt]

  \node[vertex, label=above:\tiny{v}] (r) at (2,3) {};
  \node[vertex, label=left:\tiny{}](u1) at (1,2){};
  \node[vertex, label=left:\tiny{}](u2) at (2,2){};
  \node[vertex, label=left:\tiny{}](u3) at (3,2){};
  \node[vertex, label=left:\tiny{}](u4) at (1,1){};
  \node[vertex, label=left:\tiny{}](u5) at (2,1){};
  \node[vertex, label=left:\tiny{}](u6) at (3,1){};

  \node[vertex, label=left:\tiny{}](v1) at (0,3){};
  \node[vertex, label=left:\tiny{}](v2) at (1,3){};
  \node[vertex, label=left:\tiny{}](v3) at (3,3){};
  \node[vertex, label=left:\tiny{}](v4) at (4,3){};

 \draw[color=black] 
   (r)--(u1)--(u4) (r)--(u2)--(u5) (r)--(u3)--(u6)
   (v1)--(v2)--(r)--(v3)--(v4);

\node at (1,4) {$T$};

\node at (5,2.5) {$\longrightarrow$};

\node at (6,4) {$T_0$};

\node at (6,1) {$T_1$};

  \node[root, label=above:\tiny{$v_1$}] (v) at (9,2) {};
  \node[vertex, label=left:\tiny{}](u1) at (8,1){};
  \node[vertex, label=left:\tiny{}](u2) at (9,1){};
  \node[vertex, label=left:\tiny{}](u3) at (10,1){};
  \node[vertex, label=left:\tiny{}](u4) at (8,0){};
  \node[vertex, label=left:\tiny{}](u5) at (9,0){};
  \node[vertex, label=left:\tiny{}](u6) at (10,0){};

  \node[vertex, label=above:\tiny{v}] (r) at (9,4) {};
  \node[vertex, label=left:\tiny{}](v1) at (7,4){};
  \node[vertex, label=left:\tiny{}](v2) at (8,4){};
  \node[vertex, label=left:\tiny{}](v3) at (10,4){};
  \node[vertex, label=left:\tiny{}](v4) at (11,4){};

 \draw[color=black] 
   (v)--(u1)--(u4) (v)--(u2)--(u5) (v)--(u3)--(u6)
   (v1)--(v2)--(r)--(v3)--(v4);

\end{tikzpicture}
\caption{Splitting a Tree}
\end{center}

\end {figure}
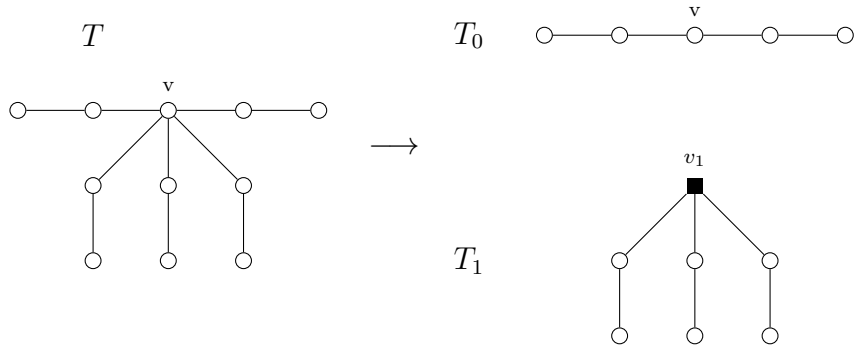

Note that when splitting a tree, one has the choice of placing branches starting at the root, in $T_0$ or in a rooted tree. 

\begin{figure}[ht]
\begin{center}
\begin{tikzpicture}[scale=1]
\tikzstyle{vertex}=[circle, draw, inner sep=0pt, minimum size=6pt]
\tikzstyle{vert}=[circle,fill=black,inner sep=3pt]
\tikzstyle{overt}=[circle,fill=black!30, inner sep=3pt]
\tikzstyle{root}=[rectangle,fill=black,inner sep=3pt]

  \node[vertex, label=above:\tiny{v}] (r) at (2,3) {};
  \node[vertex, label=left:\tiny{}](u1) at (1,2){};
  \node[vertex, label=left:\tiny{}](u2) at (2,2){};
  \node[vertex, label=left:\tiny{}](u3) at (3,2){};
  \node[vertex, label=left:\tiny{}](u4) at (1,1){};
  \node[vertex, label=left:\tiny{}](u5) at (2,1){};
  \node[vertex, label=left:\tiny{}](u6) at (3,1){};

  \node[vertex, label=left:\tiny{}](v1) at (0,3){};
  \node[vertex, label=left:\tiny{}](v2) at (1,3){};
  \node[vertex, label=left:\tiny{}](v3) at (3,3){};
  \node[vertex, label=left:\tiny{}](v4) at (4,3){};

 \draw[color=black] 
   (r)--(u1)--(u4) (r)--(u2)--(u5) (r)--(u3)--(u6)
   (v1)--(v2)--(r)--(v3)--(v4);

\node at (1,4) {$T$};

\node at (5,2.5) {$\longrightarrow$};

\node at (6,4) {$T_0$};

\node at (6,1) {$T_1$};

  \node[root, label=above:\tiny{$v_1$}] (v) at (9,2) {};
  \node[vertex, label=left:\tiny{}](u1) at (8.5,1){};
  \node[vertex, label=left:\tiny{}](u2) at (9.5,1){};
  \node[vertex, label=left:\tiny{}](u4) at (8.5,0){};
  \node[vertex, label=left:\tiny{}](u5) at (9.5,0){};

  \node[vertex, label=above:\tiny{v}] (r) at (9,4) {};
  \node[vertex, label=left:\tiny{}](v1) at (7,4){};
  \node[vertex, label=left:\tiny{}](v2) at (8,4){};
  \node[vertex, label=left:\tiny{}](v3) at (10,4){};
  \node[vertex, label=left:\tiny{}](v4) at (11,4){};
  \node[vertex, label=left:\tiny{}](w1) at (10,3){};
  \node[vertex, label=left:\tiny{}](w2) at (11,3){};

 \draw[color=black] 
   (v)--(u1)--(u4) (v)--(u2)--(u5)  (r)--(w1)--(w2)
   (v1)--(v2)--(r)--(v3)--(v4);

\end{tikzpicture}
\caption{Another Splitting}
\end{center}

\end {figure}
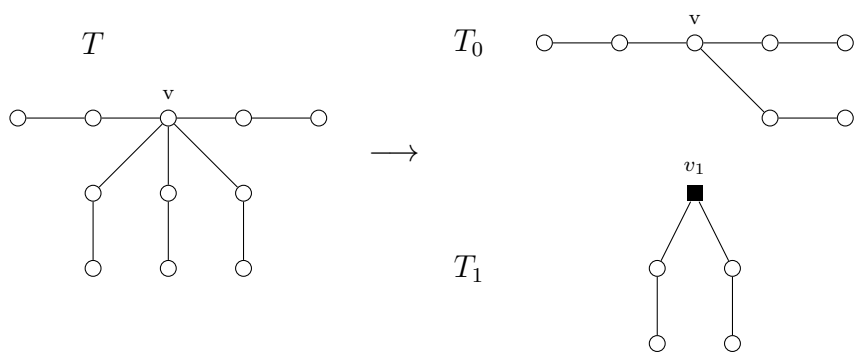

\pagebreak

\begin{prop}\label{split5}
Every tree on at least $5$ vertices can be split into a tree and either
\begin{enumerate}[label=(\roman*)]
\item a rooted tree with five vertices  
\item the rooted tree taken from $\{T'', T''', T^{iv}, T^v\}$ 
\item the rooted forest $F'$
\end{enumerate}
\end{prop}

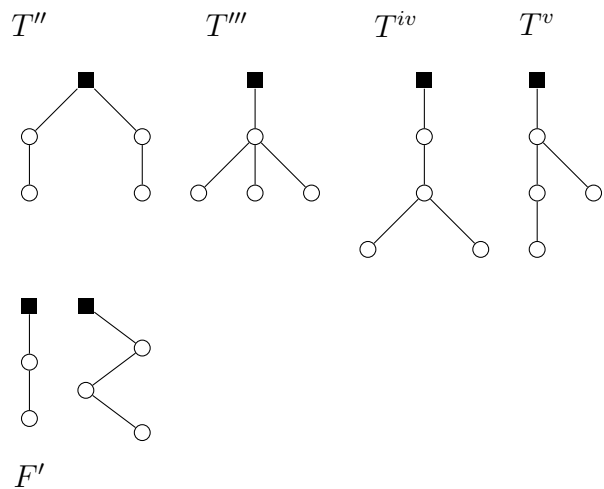
\begin{figure}[ht]
\begin{center}
\begin{tikzpicture}[scale=.75]
\tikzstyle{vertex}=[circle, draw, inner sep=0pt, minimum size=6pt]
\tikzstyle{vert}=[circle,fill=black,inner sep=3pt]
\tikzstyle{overt}=[circle,fill=black!30, inner sep=3pt]
\tikzstyle{root}=[rectangle,fill=black,inner sep=3pt]

  \node[root, label=above:\tiny{}] (r) at (2,3) {};
  \node[vertex, label=left:\tiny{}](u1) at (1,2){};
  \node[vertex, label=left:\tiny{}](u3) at (3,2){};
  \node[vertex, label=left:\tiny{}](u4) at (1,1){};
  \node[vertex, label=left:\tiny{}](u6) at (3,1){};

 \draw[color=black] 
   (r)--(u1)--(u4)  (r)--(u3)--(u6);

\node at (1,4) {$T''$};

  \node[root, label=above:\tiny{}] (r) at (5,3) {};
  \node[vertex, label=left:\tiny{}](u1) at (5,2){};
  \node[vertex, label=left:\tiny{}](u2) at (4,1){};
  \node[vertex, label=left:\tiny{}](u3) at (5,1){};
  \node[vertex, label=left:\tiny{}](u4) at (6,1){};

 \draw[color=black] 
   (r)--(u1)--(u2) (u1)--(u3) (u1)--(u4);

\node at (4.5,4) {$T'''$};

  \node[root, label=above:\tiny{}] (r) at (8,3) {};
  \node[vertex, label=left:\tiny{}](u1) at (8,2){};
  \node[vertex, label=left:\tiny{}](u2) at (8,1){};
  \node[vertex, label=left:\tiny{}](u3) at (7,0){};
  \node[vertex, label=left:\tiny{}](u4) at (9,0){};

 \draw[color=black] 
   (r)--(u1)--(u2)--(u3) (u2)--(u4);

\node at (7.5,4) {$T^{iv}$};

  \node[root, label=above:\tiny{}] (r) at (10,3) {};
  \node[vertex, label=left:\tiny{}](u1) at (10,2){};
  \node[vertex, label=left:\tiny{}](u2) at (10,1){};
  \node[vertex, label=left:\tiny{}](u3) at (11,1){};
  \node[vertex, label=left:\tiny{}](u4) at (10,0){};

 \draw[color=black] 
   (r)--(u1)--(u2)--(u4) (u1)--(u3);

\node at (10,4) {$T^v$};

  \node[root, label=above:\tiny{}] (r1) at (1,-1) {};
  \node[vertex, label=left:\tiny{}](u1) at (1,-2){};
  \node[vertex, label=left:\tiny{}](u2) at (1,-3){};

  \node[root, label=left:\tiny{}](r2) at (2,-1){};
  \node[vertex, label=left:\tiny{}](u4) at (3,-1.75){};
  \node[vertex, label=left:\tiny{}](u5) at (2,-2.5){};
  \node[vertex, label=left:\tiny{}](u6) at (3,-3.25){};

 \draw[color=black] 
   (r1)--(u1)--(u2)  (r2)--(u4)--(u5)--(u6);

\node at (1,-4) {$F'$};

\end{tikzpicture}
\end{center}

\end {figure}

\begin{proof}
Let $T$ be a tree on at least $5$ vertices and let $P=\{v_0,v_1,\dots, v_p\}$ be a longest path in $T$. We attempt to split $T$ at $v_i$ for $i=1,2,3,4,5$. Notice that if $T$ cannot be split to a rooted tree with $5$ vertices, then for some $i$, $1\leq i \leq 4$, there is a split at $v_i$ that produces a rooted tree that has less than $5$ vertices, and every split at $v_{i+1}$ produces a rooted tree with more than $5$ vertices. We call such a pair of vertices with indices $(i,i+1)$ \emph{critical}, such a split at $v_i$ \emph{deficient} and such a split at $v_{i+1}$ \emph{excessive}.

We only need consider critical pairs of vertices with indices $(2,3), (3,4),$ and $(4,5)$. For the $(2,3)$ critical pair, it is easy to see that one can produce the rooted trees $T''$ and $T'''$ by splitting at $v_2$. For the $(3,4)$ critical pair, we can produce $T^{iv}$ and $T^v$ by splitting at $v_3$. For the $(4,5)$ critical pair, one can produce rooted forest $F'$ by splitting at $v_4$.
\end{proof}


\begin{prop}\label{split6}
Every tree on at least $6$ vertices can be split into a tree and either
\begin{enumerate}[label=(\roman*)]
\item a rooted tree with six vertices  
\item the rooted tree taken from $\{T', T'_2, T'_3,T'_4\}$ 
\item the rooted forest taken from $\{F, F_2,F_3,F_4\}$
\end{enumerate}
\end{prop}

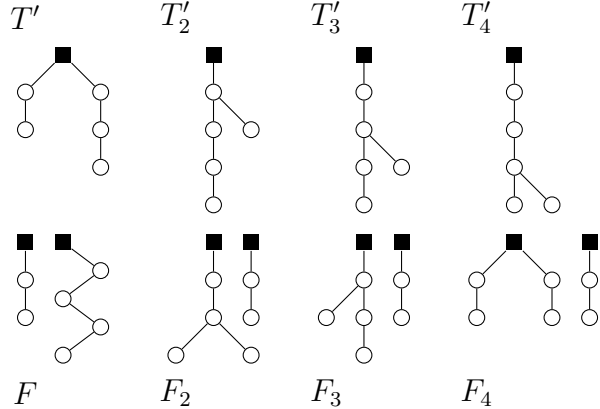
\begin{figure}[ht]
\begin{center}
\begin{tikzpicture}[scale=.5]
\tikzstyle{vertex}=[circle, draw, inner sep=0pt, minimum size=6pt]
\tikzstyle{vert}=[circle,fill=black,inner sep=3pt]
\tikzstyle{overt}=[circle,fill=black!30, inner sep=3pt]
\tikzstyle{root}=[rectangle,fill=black,inner sep=3pt]

  \node[root, label=above:\tiny{}] (r) at (2,7) {};
  \node[vertex, label=left:\tiny{}](u1) at (1,6){};
  \node[vertex, label=left:\tiny{}](u3) at (3,6){};
  \node[vertex, label=left:\tiny{}](u4) at (1,5){};
  \node[vertex, label=left:\tiny{}](u6) at (3,5){};
  \node[vertex, label=left:\tiny{}](u7) at (3,4){};

 \draw[color=black] 
   (r)--(u1)--(u4)  (r)--(u3)--(u6)--(u7);

\node at (1,8) {$T'$};

  \node[root, label=above:\tiny{}] (r) at (6,7) {};
  \node[vertex, label=left:\tiny{}](u1) at (6,6){};
  \node[vertex, label=left:\tiny{}](u2) at (6,5){};
  \node[vertex, label=left:\tiny{}](u3) at (6,4){};
  \node[vertex, label=left:\tiny{}](u4) at (6,3){};
  \node[vertex, label=left:\tiny{}](u5) at (7,5){};

 \draw[color=black] 
   (r)--(u1)--(u2)--(u3)--(u4) (u1)--(u5);

\node at (5,8) {$T'_2$};

  \node[root, label=above:\tiny{}] (r) at (10,7) {};
  \node[vertex, label=left:\tiny{}](u1) at (10,6){};
  \node[vertex, label=left:\tiny{}](u2) at (10,5){};
  \node[vertex, label=left:\tiny{}](u3) at (10,4){};
  \node[vertex, label=left:\tiny{}](u4) at (10,3){};
  \node[vertex, label=left:\tiny{}](u5) at (11,4){};

 \draw[color=black] 
   (r)--(u1)--(u2)--(u3)--(u4) (u2)--(u5);

\node at (9,8) {$T'_3$};

  \node[root, label=above:\tiny{}] (r) at (14,7) {};
  \node[vertex, label=left:\tiny{}](u1) at (14,6){};
  \node[vertex, label=left:\tiny{}](u2) at (14,5){};
  \node[vertex, label=left:\tiny{}](u3) at (14,4){};
  \node[vertex, label=left:\tiny{}](u4) at (14,3){};
  \node[vertex, label=left:\tiny{}](u5) at (15,3){};

 \draw[color=black] 
   (r)--(u1)--(u2)--(u3)--(u4) (u3)--(u5);

\node at (13,8) {$T'_4$};

  \node[root, label=above:\tiny{}] (r1) at (1,2) {};
  \node[vertex, label=left:\tiny{}](u1) at (1,1){};
  \node[vertex, label=left:\tiny{}](u2) at (1,0){};

  \node[root, label=left:\tiny{}](r2) at (2,2){};
  \node[vertex, label=left:\tiny{}](u4) at (3,1.25){};
  \node[vertex, label=left:\tiny{}](u5) at (2,.5){};
  \node[vertex, label=left:\tiny{}](u6) at (3,-.25){};
  \node[vertex, label=left:\tiny{}](u7) at (2,-1){};

 \draw[color=black] 
   (r1)--(u1)--(u2)  (r2)--(u4)--(u5)--(u6)--(u7);

\node at (1,-2) {$F$};

  \node[root, label=above:\tiny{}] (r) at (6,2) {};
  \node[vertex, label=left:\tiny{}](u3) at (6,1){};
  \node[vertex, label=left:\tiny{}](u4) at (5,-1){};
  \node[vertex, label=left:\tiny{}](u6) at (6,0){};
  \node[vertex, label=left:\tiny{}](u7) at (7,-1){};

 \draw[color=black] 

  (r)--(u3)--(u6)--(u7) (u6)--(u4);

  \node[root, label=above:\tiny{}] (r1) at (7,2) {};
  \node[vertex, label=left:\tiny{}](u1) at (7,1){};
  \node[vertex, label=left:\tiny{}](u2) at (7,0){};

 \draw[color=black] 
   (r1)--(u1)--(u2) ;

\node at (5,-2) {$F_2$};

  \node[root, label=above:\tiny{}] (r) at (10,2) {};
  \node[vertex, label=left:\tiny{}](u1) at (10,1){};
  \node[vertex, label=left:\tiny{}](u2) at (9,0){};
  \node[vertex, label=left:\tiny{}](u3) at (10,0){};
  \node[vertex, label=left:\tiny{}](u4) at (10,-1){};

 \draw[color=black] 

  (r)--(u1)--(u3)--(u4) (u1)--(u2);

  \node[root, label=above:\tiny{}] (r1) at (11,2) {};
  \node[vertex, label=left:\tiny{}](u1) at (11,1){};
  \node[vertex, label=left:\tiny{}](u2) at (11,0){};

 \draw[color=black] 
   (r1)--(u1)--(u2) ;

\node at (9,-2) {$F_3$};

  \node[root, label=above:\tiny{}] (r) at (14,2) {};
  \node[vertex, label=left:\tiny{}](u1) at (13,1){};
  \node[vertex, label=left:\tiny{}](u2) at (15,1){};
  \node[vertex, label=left:\tiny{}](u3) at (13,0){};
  \node[vertex, label=left:\tiny{}](u4) at (15,0){};

 \draw[color=black] 

  (r)--(u1)--(u3) (r)--(u2)--(u4);

  \node[root, label=above:\tiny{}] (r1) at (16,2) {};
  \node[vertex, label=left:\tiny{}](u1) at (16,1){};
  \node[vertex, label=left:\tiny{}](u2) at (16,0){};

 \draw[color=black] 
   (r1)--(u1)--(u2) ;

\node at (13,-2) {$F_4$};

\end{tikzpicture}
\end{center}

\end{figure}

\begin{proof}
The argument is similar to that in Proposition \ref{split5}. Let \linebreak$P=\{v_0,\dots, v_p\}$ be a longest path of $T$ and consider critical pairs with indices $(2,3),(3,4),(4,5),$ and $(5,6)$. There can be no critical pair with indices $(2,3)$. Critical pairs with indices $(3,4)$ produce either $T', F_2, F_3,$ or $F_4$ when splitting at $v_3$. Critical pairs with indices $(4,5)$  produce $F, T'_2,T'_3,$ or $T'_4$ when splitting the tree at $v_4$. Critical pairs with indices $(5,6)$ produce $F$ when splitting the tree at $v_5$.

\end{proof}

\section{Six-Cordial Trees}

Theorem \ref{test} implies that for any integer $k>0$, if all trees and rooted forests with $k$ vertices are $k$-cordial, then all trees are $k$-cordial. This yields a method to check if trees are $k$-cordial for any integer $k$. We shorten and simplify this test for the case $k=6$.

\begin{thm}
Every tree is $6$-cordial
\end{thm}

\begin{proof}
We induct on the order of any tree $T$. If $|T|\leq 6$, then by Proposition \ref{fact1} and Theorem \ref{cat}, $T$ is $6$-cordial. Next suppose all trees on $6(m-1)$ vertices are $6$-cordial. By Lemma \ref{hovey} it is enough to show that all trees on $6(m-1)+5$ vertices and all trees on $6m$ vertices are $6$-cordial. Let $T$ be a tree on $6m$ vertices. We apply Proposition \ref{split6} and consider case $(i)$. That is, suppose $T$ splits into a tree $T_0$ on $6(m-1)$ vertices and a rooted tree $T_1$ with root $v$, on six vertices. We consider the degrees of $v$ in $T_0$.

If $\deg(v)=1$, we can apply Lemma  \ref{add} to ``rotate'' the labels of $T_0$ so that the root vertex takes the label 0. If $w$ is the minority weight of $T_0$, we label the caterpillar $T_1$ by Grace's algorithm with the root neighbor labeled $w$.

If $\deg(v)=2$, we apply Proposition \ref{fact2} and notice that we only need to show the following rooted trees $6$-cordial:


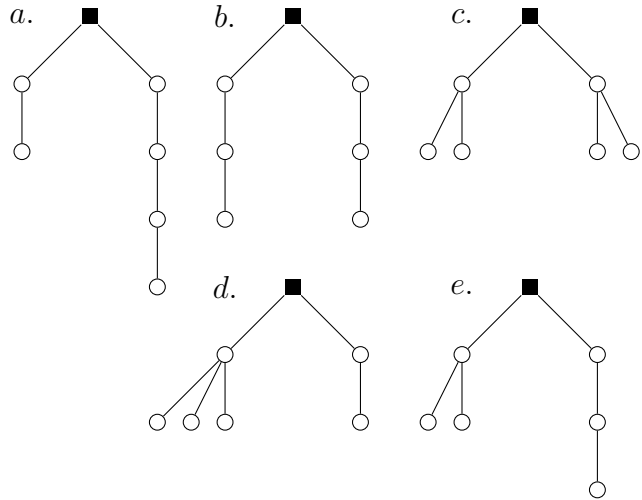
\begin{figure}[ht]
\begin{center}
\begin{tikzpicture}[scale=.9]
\tikzstyle{vertex}=[circle, draw, inner sep=0pt, minimum size=6pt]
\tikzstyle{vert}=[circle,fill=black,inner sep=3pt]
\tikzstyle{overt}=[circle,fill=black!30, inner sep=3pt]
\tikzstyle{root}=[rectangle,fill=black,inner sep=3pt]

  \node[root, label=above:\tiny{}] (r) at (2,3) {};
  \node[vertex, label=left:\tiny{}](u1) at (1,2){};
  \node[vertex, label=left:\tiny{}](u3) at (3,2){};
  \node[vertex, label=left:\tiny{}](u4) at (1,1){};
  \node[vertex, label=left:\tiny{}](u6) at (3,1){};
  \node[vertex, label=left:\tiny{}](u7) at (3,0){};
  \node[vertex, label=left:\tiny{}](u8) at (3,-1){};

 \draw[color=black] 
   (r)--(u1)--(u4)  (r)--(u3)--(u6)--(u7)--(u8);

\node at (1,3) {$a.$};

  \node[root, label=above:\tiny{}] (r) at (5,3) {};
  \node[vertex, label=left:\tiny{}](u1) at (4,2){};
  \node[vertex, label=left:\tiny{}](u3) at (6,2){};
  \node[vertex, label=left:\tiny{}](u4) at (4,1){};
  \node[vertex, label=left:\tiny{}](u6) at (6,1){};
  \node[vertex, label=left:\tiny{}](u7) at (6,0){};
  \node[vertex, label=left:\tiny{}](u8) at (4,0){};

 \draw[color=black] 
   (r)--(u1)--(u4)--(u8)  (r)--(u3)--(u6)--(u7);

\node at (4,3) {$b.$};

  \node[root, label=above:\tiny{}] (r) at (8.5,3) {};
  \node[vertex, label=left:\tiny{}](u1) at (7.5,2){};
  \node[vertex, label=left:\tiny{}](u3) at (9.5,2){};
  \node[vertex, label=left:\tiny{}](u4) at (7.5,1){};
  \node[vertex, label=left:\tiny{}](u6) at (9.5,1){};
  \node[vertex, label=left:\tiny{}](u7) at (7,1){};
  \node[vertex, label=left:\tiny{}](u8) at (10,1){};

 \draw[color=black] 
  (r)--(u1)--(u4)  (u1)--(u7)  (r)--(u3)--(u6) (u3)--(u8);

\node at (7.5,3) {$c.$};

  \node[root, label=above:\tiny{}] (r) at (5,-1) {};
  \node[vertex, label=left:\tiny{}](u1) at (4,-2){};
  \node[vertex, label=left:\tiny{}](u3) at (6,-2){};
  \node[vertex, label=left:\tiny{}](u4) at (4,-3){};
  \node[vertex, label=left:\tiny{}](u6) at (6,-3){};
  \node[vertex, label=left:\tiny{}](u7) at (3.5,-3){};
  \node[vertex, label=left:\tiny{}](u8) at (3,-3){};

 \draw[color=black] 
   (r)--(u1)--(u4)  (r)--(u3)--(u6) (u1)--(u7) (u1)--(u8);

\node at (4,-1) {$d.$};

  \node[root, label=above:\tiny{}] (r) at (8.5,-1) {};
  \node[vertex, label=left:\tiny{}](u1) at (7.5,-2){};
  \node[vertex, label=left:\tiny{}](u3) at (9.5,-2){};
  \node[vertex, label=left:\tiny{}](u4) at (7.5,-3){};
  \node[vertex, label=left:\tiny{}](u6) at (9.5,-3){};
  \node[vertex, label=left:\tiny{}](u7) at (7,-3){};
  \node[vertex, label=left:\tiny{}](u8) at (9.5,-4){};

 \draw[color=black] 
  (r)--(u1)--(u4)  (u1)--(u7)  (r)--(u3)--(u6)--(u8);

\node at (7.5,-1) {$e.$};

\end{tikzpicture}
\end{center}

\end {figure}

If $\deg(v)=3$, we apply Proposition \ref{fact2} and notice that we only need to show the following rooted trees $6$-cordial: 

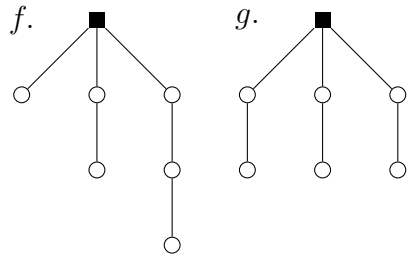
\begin{figure}[ht]
\begin{center}
\begin{tikzpicture}[scale=1]
\tikzstyle{vertex}=[circle, draw, inner sep=0pt, minimum size=6pt]
\tikzstyle{vert}=[circle,fill=black,inner sep=3pt]
\tikzstyle{overt}=[circle,fill=black!30, inner sep=3pt]
\tikzstyle{root}=[rectangle,fill=black,inner sep=3pt]

  \node[root, label=above:\tiny{}] (r) at (2,3) {};
  \node[vertex, label=left:\tiny{}](u1) at (1,2){};
  \node[vertex, label=left:\tiny{}](u3) at (3,2){};
  \node[vertex, label=left:\tiny{}](u4) at (2,2){};
  \node[vertex, label=left:\tiny{}](u6) at (3,1){};
  \node[vertex, label=left:\tiny{}](u7) at (3,0){};
  \node[vertex, label=left:\tiny{}](u8) at (2,1){};

 \draw[color=black] 
   (r)--(u1) (r)--(u4)--(u8)  (r)--(u3)--(u6)--(u7);

\node at (1,3) {$f.$};

  \node[root, label=above:\tiny{}] (r) at (5,3) {};
  \node[vertex, label=left:\tiny{}](u1) at (4,2){};
  \node[vertex, label=left:\tiny{}](u3) at (6,2){};
  \node[vertex, label=left:\tiny{}](u4) at (4,1){};
  \node[vertex, label=left:\tiny{}](u6) at (6,1){};
  \node[vertex, label=left:\tiny{}](u7) at (5,2){};
  \node[vertex, label=left:\tiny{}](u8) at (5,1){};

 \draw[color=black] 
   (r)--(u1)--(u4)  (r)--(u3)--(u6) (r)--(u7)--(u8);

\node at (4,3) {$g.$};

\end{tikzpicture}
\end{center}

\end {figure}

If $\deg(v)=4,5,6$, we apply Proposition \ref{fact2} and notice that we only need to show the following rooted tree $6$-cordial:

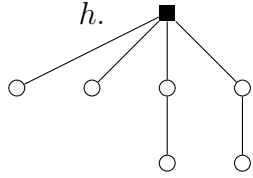
\begin{figure}[ht]
\begin{center}
\begin{tikzpicture}[scale=1]
\tikzstyle{vertex}=[circle, draw, inner sep=0pt, minimum size=6pt]
\tikzstyle{vert}=[circle,fill=black,inner sep=3pt]
\tikzstyle{overt}=[circle,fill=black!30, inner sep=3pt]
\tikzstyle{root}=[rectangle,fill=black,inner sep=3pt]

  \node[root, label=above:\tiny{}] (r) at (2,3) {};
  \node[vertex, label=left:\tiny{}](u1) at (1,2){};
  \node[vertex, label=left:\tiny{}](u3) at (3,2){};
  \node[vertex, label=left:\tiny{}](u4) at (2,2){};
  \node[vertex, label=left:\tiny{}](u6) at (3,1){};
  \node[vertex, label=left:\tiny{}](u7) at (2,1){};
  \node[vertex, label=left:\tiny{}](u8) at (0,2){};

 \draw[color=black] 
   (r)--(u1) (r)--(u4)--(u7)  (r)--(u3)--(u6) (r)--(u8);

\node at (1,3) {$h.$};

\end{tikzpicture}
\end{center}

\end {figure}

For the representations of the rooted trees $a$ through $h$ above, we define the level $\ell_0$ as the root. For any $i>0$, the level $\ell_i$ is defined to be those vertices of distance $i$ from the root, ordered from left to right as in the representation above. Using this notation, we provide a $6$-cordial labeling for each rooted tree, so that either every weight appears once or for every weight there is a labeling of  each rooted tree with that weight serving as a majority weight. For each case above, our labeling will be on vertices starting with the root and continuing to subsequent levels, from left to right. We end each labeling by listing which weight is in the majority. These labeling can be found in List $1$ at the end of the proof.
\medskip

Suppose $T$ splits into $T_0$ on $6(m-1)$ vertices and one of the above rooted trees $T_1$. By induction, we can label $T_0$ $6$-cordially with some minority weight $w$. However, we have shown that no matter the value of $w$, we can make $w$ a majority weight of $T_1$ or there is a labeling of $T_1$ with no majority weight. Pasting $T$ back together with this labeling produces a $6$-cordial labeling of $T$.

Next, we consider case $(ii)$ of Proposition \ref{split6}. We label the rooted trees $T',T'_2,T'_3,T'_4$ so that no label or weight appears more than once, and for every label there is a labeling with that label not present. These labelings can be found in List $2$ at the end of the proof.

Since $T_0$ is of order $6m+1$, we can label it $6$-cordially by the induction hypothesis and Hovey's lemma. This means that no weight on $T_0$ is in the minority and one label is in the majority. For the minority label in one of the cordial labelings of the rooted tree $T'$, we choose the majority label of $T_0$. Pasting $T$ back together produces a $6$-cordial labeling of $T$.

We now consider case $(iii)$ of Proposition \ref{split6}. We label the rooted forests $F,F_2,F_3,F_4$ so that zero appears on the left root and consider all other possible labels on the other root. These labelings can be found in List $3$ at the end of the proof.

For trees on $6(m-1)+5$ vertices, we can use many of the previous labelings on trees with $6m$ vertices. We argue as in the case of $6m$ vertices.

Let $T$ be a tree on $6(m-1)+5$ vertices. We apply Proposition \ref{split5} and consider case $(i)$. That is, suppose $T$ splits into a tree $T_0$ on $6(m-1)$ vertices and a rooted tree $T_1$ with root $v$, on five vertices. We consider the degrees of $v$ in $T_0$.

If $\deg(v)=1$, we can apply Lemma  \ref{add} to ``rotate'' the labels of $T_0$ so that the root vertex takes the label 0. If $w$ is the minority weight of $T_0$, we label the caterpillar $T_1$ by Grace's algorithm with the root neighbor labeled $w$.

If $\deg(v)=2$, we consider the following trees:

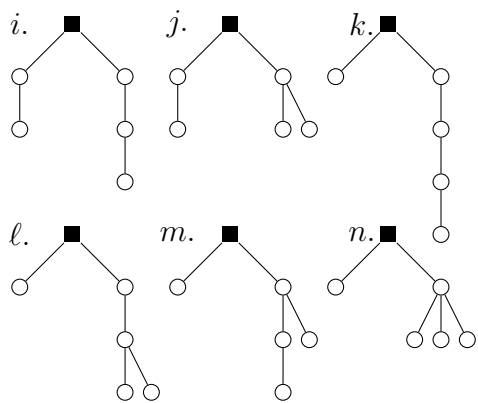
\begin{figure}[ht]
\begin{center}
\begin{tikzpicture}[scale=.7]
\tikzstyle{vertex}=[circle, draw, inner sep=0pt, minimum size=6pt]
\tikzstyle{vert}=[circle,fill=black,inner sep=3pt]
\tikzstyle{overt}=[circle,fill=black!30, inner sep=3pt]
\tikzstyle{root}=[rectangle,fill=black,inner sep=3pt]

  \node[root, label=above:\tiny{}] (r) at (2,3) {};
  \node[vertex, label=left:\tiny{}](u1) at (1,2){};
  \node[vertex, label=left:\tiny{}](u3) at (3,2){};
  \node[vertex, label=left:\tiny{}](u4) at (1,1){};
  \node[vertex, label=left:\tiny{}](u6) at (3,1){};
  \node[vertex, label=left:\tiny{}](u7) at (3,0){};

 \draw[color=black] 
   (r)--(u1)--(u4)  (r)--(u3)--(u6)--(u7);

\node at (1,3) {$i.$};

  \node[root, label=above:\tiny{}] (r) at (5,3) {};
  \node[vertex, label=left:\tiny{}](u1) at (4,2){};
  \node[vertex, label=left:\tiny{}](u3) at (6,2){};
  \node[vertex, label=left:\tiny{}](u4) at (4,1){};
  \node[vertex, label=left:\tiny{}](u6) at (6,1){};
  \node[vertex, label=left:\tiny{}](u7) at (6.5,1){};

 \draw[color=black] 
   (r)--(u1)--(u4)  (r)--(u3)--(u6) (u3)--(u7);

\node at (4,3) {$j.$};

  \node[root, label=above:\tiny{}] (r) at (8,3) {};
  \node[vertex, label=left:\tiny{}](u1) at (7,2){};
  \node[vertex, label=left:\tiny{}](u3) at (9,2){};
  \node[vertex, label=left:\tiny{}](u6) at (9,1){};
  \node[vertex, label=left:\tiny{}](u7) at (9,0){};
  \node[vertex, label=left:\tiny{}](u8) at (9,-1){};

 \draw[color=black] 
   (r)--(u1)  (r)--(u3)--(u6)--(u7)--(u8);

\node at (7.5,3) {$k.$};

  \node[root, label=above:\tiny{}] (r) at (2,-1) {};
  \node[vertex, label=left:\tiny{}](u1) at (1,-2){};
  \node[vertex, label=left:\tiny{}](u3) at (3,-2){};
  \node[vertex, label=left:\tiny{}](u6) at (3,-3){};
  \node[vertex, label=left:\tiny{}](u7) at (3,-4){};
  \node[vertex, label=left:\tiny{}](u8) at (3.5,-4){};

 \draw[color=black] 
   (r)--(u1)  (r)--(u3)--(u6)--(u7) (u6)--(u8);

\node at (1,-1) {$\ell.$};

  \node[root, label=above:\tiny{}] (r) at (5,-1) {};
  \node[vertex, label=left:\tiny{}](u1) at (4,-2){};
  \node[vertex, label=left:\tiny{}](u3) at (6,-2){};
  \node[vertex, label=left:\tiny{}](u6) at (6,-3){};
  \node[vertex, label=left:\tiny{}](u7) at (6,-4){};
  \node[vertex, label=left:\tiny{}](u8) at (6.5,-3){};

 \draw[color=black] 
   (r)--(u1)  (r)--(u3)--(u6)--(u7) (u3)--(u8);

\node at (4,-1) {$m.$};

  \node[root, label=above:\tiny{}] (r) at (8,-1) {};
  \node[vertex, label=left:\tiny{}](u1) at (7,-2){};
  \node[vertex, label=left:\tiny{}](u3) at (9,-2){};
  \node[vertex, label=left:\tiny{}](u6) at (9,-3){};
  \node[vertex, label=left:\tiny{}](u7) at (8.5,-3){};
  \node[vertex, label=left:\tiny{}](u8) at (9.5,-3){};

 \draw[color=black] 
   (r)--(u1)  (r)--(u3)--(u6) (u7)--(u3)--(u8);

\node at (7.5,-1) {$n.$};

\end{tikzpicture}
\end{center}

\end {figure}

If $\deg(v)=3$, we consider the next two trees:

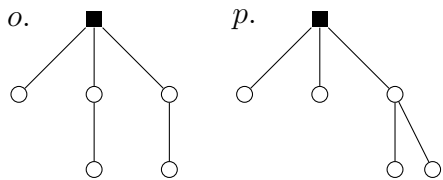
\begin{figure}[ht]
\begin{center}
\begin{tikzpicture}[scale=1]
\tikzstyle{vertex}=[circle, draw, inner sep=0pt, minimum size=6pt]
\tikzstyle{vert}=[circle,fill=black,inner sep=3pt]
\tikzstyle{overt}=[circle,fill=black!30, inner sep=3pt]
\tikzstyle{root}=[rectangle,fill=black,inner sep=3pt]

  \node[root, label=above:\tiny{}] (r) at (2,3) {};
  \node[vertex, label=left:\tiny{}](u1) at (1,2){};
  \node[vertex, label=left:\tiny{}](u2) at (2,2){};
  \node[vertex, label=left:\tiny{}](u3) at (3,2){};
  \node[vertex, label=left:\tiny{}](u4) at (2,1){};
  \node[vertex, label=left:\tiny{}](u6) at (3,1){};

 \draw[color=black] 
   (r)--(u1) (r)--(u2)--(u4)  (r)--(u3)--(u6);

\node at (1,3) {$o.$};

  \node[root, label=above:\tiny{}] (r) at (5,3) {};
  \node[vertex, label=left:\tiny{}](u1) at (4,2){};
  \node[vertex, label=left:\tiny{}](u2) at (5,2){};

  \node[vertex, label=left:\tiny{}](u3) at (6,2){};
  \node[vertex, label=left:\tiny{}](u4) at (6,1){};
  \node[vertex, label=left:\tiny{}](u6) at (6.5,1){};


 \draw[color=black] 
   (r)--(u1) (r)--(u2)  (r)--(u3)--(u4) (u3)--(u6);

\node at (4,3) {$p.$};

\end{tikzpicture}
\end{center}

\end {figure}

\pagebreak

If $\deg(v)=4$ we only need to consider the following tree:

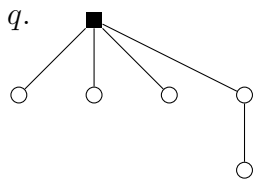
\begin{figure}[ht]
\begin{center}
\begin{tikzpicture}[scale=1]
\tikzstyle{vertex}=[circle, draw, inner sep=0pt, minimum size=6pt]
\tikzstyle{vert}=[circle,fill=black,inner sep=3pt]
\tikzstyle{overt}=[circle,fill=black!30, inner sep=3pt]
\tikzstyle{root}=[rectangle,fill=black,inner sep=3pt]

  \node[root, label=above:\tiny{}] (r) at (2,3) {};
  \node[vertex, label=left:\tiny{}](u1) at (1,2){};
  \node[vertex, label=left:\tiny{}](u2) at (2,2){};
  \node[vertex, label=left:\tiny{}](u3) at (3,2){};
  \node[vertex, label=left:\tiny{}](u4) at (4,2){};
  \node[vertex, label=left:\tiny{}](u5) at (4,1){};

 \draw[color=black] 
   (u1)--(r)--(u2) (r)--(u3)  (r)--(u4)--(u5);

\node at (1,3) {$q.$};
\end{tikzpicture}
\end{center}

\end {figure}

If $\deg(v)=5$ we have the following tree:

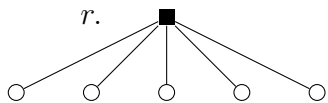
\begin{figure}[ht]
\begin{center}
\begin{tikzpicture}[scale=1]
\tikzstyle{vertex}=[circle, draw, inner sep=0pt, minimum size=6pt]
\tikzstyle{vert}=[circle,fill=black,inner sep=3pt]
\tikzstyle{overt}=[circle,fill=black!30, inner sep=3pt]
\tikzstyle{root}=[rectangle,fill=black,inner sep=3pt]

  \node[root, label=above:\tiny{}] (r) at (2,3) {};
  \node[vertex, label=left:\tiny{}](u1) at (1,2){};
  \node[vertex, label=left:\tiny{}](u2) at (2,2){};
  \node[vertex, label=left:\tiny{}](u3) at (3,2){};
  \node[vertex, label=left:\tiny{}](u4) at (4,2){};
  \node[vertex, label=left:\tiny{}](u5) at (0,2){};

 \draw[color=black] 
   (u1)--(r)--(u2) (r)--(u3)  (u5)--(r)--(u4);

\node at (1,3) {$r.$};
\end{tikzpicture}
\end{center}

\end {figure}

To show $6$-cordial labelings we argue, whenever possible, that the rooted trees above are rooted subgraphs of rooted trees on $6$ vertices for which we have already shown labelings in which each weight appeared in the majority. In this case, if $T$ splits into $T_0$ and such a rooted tree $T_1$, then $T_0$ can be labeled $6$-cordially with some minority weight $w$. We have shown that we can make $w$ a majority weight in a rooted tree $T_r$ on six vertices containing $T_1$ as a subgraph. When we remove the edge $e$ of $T_r$ to make it $T_1$, notice that if the weight on $e$ was $w$, then no weight appears in the majority in this labeling of $T_1$. If the weight on $e$ was not $w$, then this weight was not in the minority in $T_0$. In either case, pasting $T$ back together with this labeling produces a $6$-cordial labeling of $T$.

For each of the above rooted trees, we describe how it is a rooted subtree of a rooted tree which we have labeled or produce a new labeling. When we provide new labelings, we produce two $6$-cordial labelings with no repeated vertex label and one minority weight, which will be different in the two labelings. If a minority weight of $T_0$ is one of the two minority weights in the labelings of $T_1$, then we use the other. If the minority weight of $T_0$ is neither of the two minority weights in the labelings of $T_1$, we use either labeling. We follow the notation in the labelings of the previous cases. These descriptions can be found in List $4$ at the end of the proof.

Next, we consider case $(ii)$ of Proposition \ref{split5}. We label the rooted trees $T'', T''', T^{iv}, T^v$ so that no label or weight appears more than once, and for every label there is a labeling with that label not present. These labelings can be found in List $5$ at the end of the proof.

Since $T_0$ is of order $6m+1$, we can label it $6$-cordially by the induction hypothesis and Hovey's lemma. This means that no weight on $T_0$ is in the minority and one label is in the majority. For the minority label in one of the cordial labelings of the rooted tree $T'$, we choose the majority label of $T_0$. Pasting $T$ back together produces a $6$-cordial labeling of $T$.

Finally, for case $(iii)$ of Proposition \ref{split5}, notice that the rooted forest $F'$ is a rooted subforest of $F$. Hence, as in the case of the rooted trees of order $5$, we can apply the labelings of $F$.

\medskip

\underline{\emph{List $1$:}}

\noindent $a.\,\, 0,4,5,2,1,0,3;$ majority weight $=0$.\\
$a.\,\, 0,1,4,0,5,3,2;$ majority weight $=1$.\\
$a.\,\, 0,2,5,1,3,4,0;$ majority weight $=2$.\\
$a.\,\, 0,2,3,5,0,4,1;$ majority weight $=3$.\\
$a.\,\, 0,2,3,5,1,4,0;$ majority weight $=4$.\\
$a.\,\, 0,4,5,1,2,0,3;$ majority weight $=5$.\\

\noindent $b.\,\, 0,0,3,4,5,1,2;$ no majority weight.\\

\noindent $c.\,\, 0,0,3,1,4,2,5;$ no majority weight.\\

\noindent $d.\,\, 0,0,5,2,3,4,1;$ majority weight $=0$.\\
$d.\,\, 0,0,5,1,3,4,2;$ majority weight $=1$.\\
$d.\,\, 0,2,3,0,5,4,1;$ majority weight $=2$.\\
$d.\,\, 0,0,5,1,2,3,4;$ majority weight $=3$.\\
$d.\,\, 0,5,4,1,2,3,0;$ majority weight $=4$.\\
$d.\,\, 0,0,3,1,4,5,2;$ majority weight $=5$.\\

\noindent $e.\,\, 0,2,0,4,5,3,1;$ majority weight $=0$.\\
$e.\,\, 0,1,0,4,3,2,5;$ majority weight $=1$.\\
$e.\,\, 0,2,3,5,0,1,4;$ majority weight $=2$.\\
$e.\,\, 0,5,3,1,2,0,4;$ majority weight $=3$.\\
$e.\,\, 0,5,4,2,1,0,3;$ majority weight $=4$.\\
$e.\,\, 0,5,0,2,3,4,1;$ majority weight $=5$.\\

\noindent $f.\,\, 0,0,3,5,1,2,4;$ majority weight $=0$.\\
$f.\,\, 0,2,3,1,4,5,0;$ majority weight $=1$.\\
$f.\,\, 0,3,2,4,0,1,5;$ majority weight $=2$.\\
$f.\,\, 0,2,1,3,5,0,4;$ majority weight $=3$.\\
$f.\,\, 0,4,3,5,1,2,0;$ majority weight $=4$.\\
$f.\,\, 0,4,3,5,2,1,0;$ majority weight $=5$.\\

\noindent $g.\,\, 0,0,2,3,5,4,1;$ majority weight $=0$.\\
$g.\,\, 0,1,2,4,0,3,5;$ majority weight $=1$.\\
$g.\,\, 0,0,1,3,2,4,5;$ majority weight $=2$.\\
$g.\,\, 0,0,1,3,4,2,5;$ majority weight $=3$.\\
$g.\,\, 0,0,2,3,4,5,1;$ majority weight $=4$.\\
$g.\,\, 0,5,1,2,0,3,4;$ majority weight $=5$.\\

\noindent $h.\,\, 0,3,5,0,4,1,2;$ majority weight $=0$.\\
$h.\,\, 0,1,2,0,3,5,4;$ majority weight $=1$.\\
$h.\,\, 0,3,4,2,5,0,1;$ majority weight $=2$.\\
$h.\,\, 0,2,4,1,3,5,0;$ majority weight $=3$.\\
$h.\,\, 0,2,3,1,4,5,0;$ majority weight $=4$.\\
$h.\,\, 0,3,5,4,0,1,2;$ majority weight $=5$.

\medskip

\underline{\emph{List $2$:}}

\noindent $T': 0,2,3,5,1,4;$ minority label $=0$.\\
$T': 0,2,0,3,4,5;$ minority label $=1$.\\
$T': 0,1,0,3,5,4;$ minority label $=2$.\\
$T': 0,4,0,1,2,5;$ minority label $=3$.\\
$T': 0,1,0,2,5,3;$ minority label $=4$.\\
$T': 0,2,0,1,4,3;$ minority label $=5$.

\medskip

\noindent $T'_2: 0,4,2,1,5,3;$ minority label $=0$.\\
$T'_2: 0,5,3,2,0,4;$ minority label $=1$.\\
$T'_2: 0,1,5,4,3,0;$ minority label $=2$.\\
$T'_2: 0,2,4,1,0,5;$ minority label $=3$.\\
$T'_2: 0,2,3,5,0,1;$ minority label $=4$.\\
$T'_2: 0,0,3,2,1,4;$ minority label $=5$.

\medskip

\noindent $T'_3: 0,3,1,4,5,2;$ minority label $=0$.\\
$T'_3: 0,4,2,5,0,3;$ minority label $=1$.\\
$T'_3: 0,0,4,1,3,5;$ minority label $=2$.\\
$T'_3: 0,5,1,2,0,4;$ minority label $=3$.\\
$T'_3: 0,3,1,5,0,2;$ minority label $=4$.\\
$T'_3: 0,0,2,1,3,4;$ minority label $=5$.

\medskip

\noindent $T'_4: 0,4,1,5,2,3;$ minority label $=0$.\\
$T'_4: 0,5,2,0,3,4;$ minority label $=1$.\\
$T'_4: 0,5,1,0,3,4;$ minority label $=2$.\\
$T'_4: 0,5,1,0,2,4;$ minority label $=3$.\\
$T'_4: 0,5,1,0,2,3;$ minority label $=4$.\\
$T'_4: 0,4,1,0,2,3;$ minority label $=5$.

\medskip

\underline{\emph{List $3$:}}

Labels for $F:$

\noindent $0,0,4,0,1,3,5,2;$ no majority weight.\\

\noindent $0,1,0,1,3,4,2,5;$ majority weight $=0$.\\
$0,1,5,0,2,3,1,4;$ majority weight $=1$.\\
$0,1,2,5,0,4,1,3;$ majority weight $=2$.\\
$0,1,4,2,1,5,3,0;$ majority weight $=3$.\\
$0,1,4,3,1,0,2,5;$ majority weight $=4$.\\
$0,1,5,2,0,4,3,1;$ majority weight $=5$.\\

\noindent $0,2,3,4,1,2,5,0;$ majority weight $=0$.\\
$0,2,4,5,1,2,0,3;$ majority weight $=1$.\\
$0,2,1,0,3,2,4,5;$ majority weight $=2$.\\
$0,2,0,2,3,5,4,1;$ majority weight $=3$.\\
$0,2,2,4,5,0,3,1;$ majority weight $=4$.\\
$0,2,5,4,2,1,3,0;$ majority weight $=5$.\\

\noindent $0,3,0,1,3,4,2,5;$ majority weight $=0$.\\
$0,3,1,2,0,4,5,3;$ majority weight $=1$.\\
$0,3,4,5,1,2,0,3;$ majority weight $=2$.\\
$0,3,3,4,0,1,5,2;$ majority weight $=3$.\\
$0,3,4,3,0,5,2,1;$ majority weight $=4$.\\
$0,3,5,4,0,2,1,3;$ majority weight $=5$.\\

\noindent $0,4,0,2,1,3,5,4;$ majority weight $=0$.\\
$0,4,1,2,4,5,3,0;$ majority weight $=1$.\\
$0,4,2,3,0,1,5,4;$ majority weight $=2$.\\
$0,4,3,4,0,1,5,2;$ majority weight $=3$.\\
$0,4,4,0,1,3,5,2;$ majority weight $=4$.\\
$0,4,5,2,0,1,3,4;$ majority weight $=5$.\\

\noindent $0,5,0,1,2,3,4,5;$ majority weight $=0$.\\
$0,5,1,3,0,2,4,5;$ majority weight $=1$.\\
$0,5,2,3,5,0,4,1;$ majority weight $=2$.\\
$0,5,3,0,1,2,4,5;$ majority weight $=3$.\\
$0,5,4,3,0,2,5,1;$ majority weight $=4$.\\
$0,5,5,1,2,4,0,3;$ majority weight $=5$.\\

\medskip

Labels for $F_2:$

Roots $0$ and $1$:

\noindent $0, 1, 3, 0, 1, 2, 4, 5;$ no majority weight\\

Roots $0$ and $3$:

\noindent $0, 3, 5, 4, 3, 2, 0, 1;$ no majority weight\\

Roots $0$ and $5$:

\noindent $0, 5, 3, 2, 1, 0, 4, 5;$ no majority weight\\

Roots $0$ and $4$:

\noindent $0 ,4, 1, 0, 4, 2, 3, 5;$ majority weight $1$\\
$0, 4, 1, 0, 5, 2, 3, 4;$ majoirty weight $2$\\
$0, 4, 3, 0, 4, 2, 1, 5;$ majority weight $3$\\
$0, 4, 3, 0, 5, 4, 1, 2;$ majority weight $4$\\
$0, 4, 5, 0, 4, 2, 1, 3;$ majority weight $5$\\
$0, 4, 3, 1, 4, 5, 0, 2;$ majority weight $0$\\

Roots $0$ and $0$: 

\noindent $0, 0, 4, 0, 3, 1, 2, 5;$ majority weight $1$\\
$0, 0, 1, 2, 3, 4, 0, 5;$ majority weight $2$\\
$0, 0, 0, 2, 4, 1, 3, 5;$ majority weight $3$\\
$0, 0, 1, 0, 3, 4, 2, 5;$ majority weight $4$\\
$0, 0, 2, 0, 1, 5, 3, 4;$ majority weight $5$\\
$0, 0, 3, 5, 4, 1, 2, 0;$ majority weight $0$\\

Roots $0$ and $2$:

\noindent $0, 2, 1, 2, 4, 0, 5, 3;$ majority weight $1$\\
$0, 2, 3, 0, 1, 2, 4, 5;$ majority weight $2$\\
$0, 2, 3, 2, 4, 0, 5, 1;$ majority weight $3$\\
$0, 2, 1, 4, 3, 2, 0, 5;$ majority weight $4$\\
$0, 2, 5, 3, 2, 1, 0, 4;$ majority weight $5$\\
$0, 2, 3, 4, 5, 0, 1, 2;$ majority weight $0$\\

\medskip

Labels for $F_3:$

\noindent $0, 1, 0, 1, 5, 4, 2, 3;$ no majority weight\\
$0, 2, 4, 0, 1, 2, 3, 5;$ no majority weight\\
$0, 3, 0, 1, 3, 2, 4, 5;$ no majority weight\\
$0, 4, 2, 0, 5, 4, 3, 1;$ no majority weight\\
$0, 5, 4, 3 ,1, 2, 0, 5;$ no majority weight\\
$0, 0, 0, 1, 4, 5, 2, 3;$ no majority weight\\

\medskip

Labels for $F_4:$

\noindent $0, 0, 2, 4, 0, 5, 1, 3;$ no majority weight\\
$0, 1, 4, 5, 2, 3, 1, 0;$ no majority weight\\
$0, 2, 3, 2, 1, 1, 4, 5;$ no majority weight\\
$0, 3, 5, 4, 0, 1, 3, 2;$ no majority weight\\
$0, 4, 4, 3, 1, 2, 5, 0;$ no majority weight\\
$0, 5, 3, 4, 0, 5, 2, 1;$ no majority weight\\

\medskip

\underline{\emph{List $4$:}}

\noindent$i.\,\, $ This tree is a rooted subtree of the tree in $a.$\\
$j.\,\, $ This tree is a rooted subtree of the tree in $d.$\\
$k.\,\, $ This tree is a rooted subtree of the tree in $a.$\\
$\ell. \,\, 0,2,3,1,4,5;$ minority weight $=1$\\ 
$\ell. \,\, 0,4,3,5,2,1;$ minority weight $=5$\\
$m. \,\, 0,3,4,2,1,5;$ minority weight $=2$\\
$m. \,\, 0,3,2,4,5,1;$ minority weight $=4$\\
$n. \,\, $ This tree is a rooted subtree of the tree in $d.$\\
$o. \,\, $ This tree is a rooted subtree of the tree in $f.$\\
$p. \,\, 0,3,4,5,1,2;$ minority weight $=2$\\
$p. \,\, 0,3,2,1,5,4;$ minority weight $=4$\\
$q. \,\, $ This tree is a rooted subtree of the tree in $h.$\\ 
$r. \,\, 0,1,2,3,4,5;$ minority weight $=0$\\
$r. \,\, 0,0,1,2,3,4;$ minority weight $=5$\\ 

\medskip

\underline{\emph{List $5$:}}

$T'': $

\noindent$0,1,2,3,4;$ minority labels are $5,0$.\\
$0,0,1,5,2;$ minority labels are $3,4$.\\
$0,0,5,4,3;$ minority labels are $1,2$.\\

$T''': $

\noindent$0,0,1,2,3;$ minority labels are $4,5$.\\
$0,5,1,2,4;$ minority labels are $0,3$.\\
$0,0,3,4,5;$ minority labels are $1,2$.\\

$T^{iv}:$

\noindent$0,2,1,4,5;$ minority labels are $0,3$.\\
$0,0,1,4,3;$ minority labels are $2,5$.\\
$0,0,2,3,5;$ minority labels are $1,4$.\\

$T^v:$

\noindent$0,0,2,1,5;$ minority labels are $3,4$.\\
$0,1,4,3,0;$ minority labels are $2,5$.\\
$0,5,3,2,4;$ minority labels are $0,1$.\\

\end{proof}

\end{document}